   \documentclass[letterpaper, 10 pt, conference]{ieeeconf}

\IEEEoverridecommandlockouts              

\usepackage{amssymb}
\usepackage{amsmath,amssymb}
\usepackage{thmtools}
\usepackage{amsfonts} 
\usepackage{verbatim} 
\usepackage{setspace} 
\usepackage{graphicx}

\newcommand{\spanop}{\operatorname{span}}

\newtheorem{theorem}{Theorem}

\newtheorem{proposition}[theorem]{Proposition}

\newtheorem{Remark}[theorem]{Remark}

\declaretheorem[name={Example}  ] {Example}

\newcommand{\beq}       {\begin{equation}}
\newcommand{\eeq}       {\end{equation}}
\newcommand{\bery}      {\begin{array}}
\newcommand{\eery}      {\end{array}}
\newcommand{\berys}     {\begin{array*}}
\newcommand{\eerys}     {\end{array*}}
\newcommand{\beqry}     {\begin{eqnarray}}
\newcommand{\eeqry}     {\end{eqnarray}}
\newcommand{\beqrys}    {\begin{eqnarray*}}
\newcommand{\eeqrys}    {\end{eqnarray*}}

\def\bbr{{\mathbb R}}
\def \R {{\mathbb R}}

\newcommand{\M}{{\mathbf M}}

\newcommand{\be}{\begin{equation}}
\newcommand{\ee}{\end{equation}}

\usepackage{xcolor}

\begin{document}
\title{Behavior  of Totally Positive Differential Systems Near a Periodic Solution}

\author{Chengshuai Wu, Lars Gr\"une, Thomas Kriecherbauer, and Michael Margaliot\thanks{
CW and MM are  with the  School of Electrical Engineering, Tel Aviv
University, Israel. LG and TK are  with the
Mathematical Institute, University of Bayreuth, Germany. This research is partially supported by a research grant from the ISF. Correspondence: michaelm@tauex.tau.ac.il}}

\maketitle
 
\begin{abstract}
A time-varying nonlinear dynamical system is called a  totally positive differential system~(TPDS)
if its Jacobian admits  a special sign pattern: it is tri-diagonal with positive entries on the super- and sub-diagonals. If the vector field of a~TPDS is $T$-periodic then 
 every bounded trajectory  
converges to a $T$-periodic solution. In particular, when the vector field is time-invariant 
every bounded trajectory of a TPDS
converges to an equlbrium. 
Here, we use the spectral 
 theory of oscillatory matrices to analyze the behavior near a periodic solution of a TPDS.  This yields information on the perturbation directions     that lead to the fastest and slowest convergence to or divergence 
 from the periodic solution. We demonstrate   the 
   theoretical results  using   a model from  systems biology called the ribosome flow model. 
\end{abstract}


\section{Introduction}

Consider the time-varying  nonlinear system:
\be\label{eq:tpds}
\dot x(t)=f(t,x(t)),
\ee
where~$f: \R_+ \times \Omega \to \R^n$ is   continuously differentiable and~$\Omega$ is a convex subset of~$\R^n$. 
For~$t\geq t_0 \geq 0$ and~$a\in\Omega$,
let~$x(t,t_0,x_0)$ denote the solution of~\eqref{eq:tpds}
at time~$t$
with the initial condition~$x(t_0)=x_0$. 
From here on we always take~$t_0=0$ and write~$x(t,x_0)$ for~$x(t,0,x_0)$.  
We assume throughout that   for any~$x_0\in \Omega$ the solution~$x(t,x_0) \in \Omega$  for all~$t\geq 0$.

Let~$J(t,x):=\frac{\partial }{\partial x}f(t,x)$ denote the Jacobian of the vector field~$f$. 
Let~$\M^+  \subset \bbr^{n\times n}$ denote the set of $n\times n$ tri-diagonal 
matrices with positive entries of the super- and sub-diagonals, i.e. the subset of Jacobi matrices. 
 System~\eqref{eq:tpds} is called a \emph{totally positive differential system}~(TPDS) if~$J(t,x)\in \M^+ $  for all~$(t,x)$.  
In particular,~$J(t,x)$ is a  Metzler matrix (i.e. a matrix whose 
  off-diagonal entries are all non-negative), and also an irreducible matrix. This   implies that~\eqref{eq:tpds} is strongly cooperative~\cite{hlsmith}, i.e. if~$a,b\in\Omega$, with~$a\not =b$, and 
	\[
	a-b\in\bbr^n_+ :=\{x\in\bbr^n: x_i\geq 0,\; i=1,\dots,n  \}
	\]
  then
	\[
	x(t,a)-x(t,b) \in \bbr^n_{++}:= \{x\in\bbr^n: x_i >  0,\; i=1,\dots,n  \}  
	\]
	 for all $t >  0$.
Cooperative systems
play an important role in models from  biology and neuroscience, where it is known that the state-variables represent quantities that can never attain negative values, e.g. concentration of molecules, average number of spikes in a neuron, etc.
Cooperative systems enjoy
  a well-ordered behavior. By Hirsch's quasi-convergence theorem~\cite{hlsmith}, almost every bounded trajectory of a time-invariant (i.e.~$f(t,x)=f(x)$) cooperative system converges to the   set of equilibria.

 TPDSs require a stronger condition on the Jacobian 
 and consequently enjoy stronger properties. 
 Smillie~\cite{smillie} proved that if   the  vector field is time-invariant then
\emph{every} bounded trajectory   converges to  an equilibrium.
Smith~\cite{periodic_tridi_smith} generalized 
this result and showed that if
 the vector field~$f$ of a TPDS  is~$T$-periodic, that is,
 \[
 f(t,z)=f(t+T,z) \text{ for all } t \geq 0,z\in\R^n
 \]
then  {every} bounded trajectory   converges to  a~$T$-periodic solution of the system.  
These results found many 
    applications in fields such as  systems biology and 
		neuroscience~\cite{RFM_stability,chua_roska_1990,Donnell2009120}, 
as well as several  generalizations 
(see, e.g.~\cite{fgwang2013,weiss-cdc57}).

The proofs of Smillie and Smith  are based
on using the number of sign changes in the vector of derivatives~$\dot x(t)$ as an integer-valued Lyapunov function. 
Recently, it was shown that these results are closely related to the sign-variation diminishing property of totally positive matrices~\cite{fulltppaper}.
Recall that a  matrix~$A\in\bbr^{n\times n }$ is called totally non-negative~[TN]
if all its minors are non-negative, and totally positive~[TP] if they are all positive. 
Such matrices have a rich and beautiful structure~\cite{pinkus,total_book}. 
For example,  if~$A\in\bbr^{n\times n}$ is~TP and~$v\in\bbr ^n\setminus\{0\}$, then  
  the number of sign variations in~$Av$ is smaller or equal to  the number of sign variations in~$v$.
 
  Schwarz~\cite{schwarz1970} 
  introduced and analyzed \emph{linear} TPDSs. 
	He showed in particular that  the following two conditions are equivalent: 1)~$A\in \M^+$; 2) the transition matrix~$\exp(At)$ of the linear time-invariant~(LTI) system~$\dot x=Ax $
	is TP for any~$t>0$. 
	Unfortunately, his results were 
	almost forgotten.

  In many dynamical systems  it
  is important to understand
how a small perturbation affects the behavior
near a periodic solution (and, in particular, near an equilibrium). 
In general, this is a difficult problem, as 
there is little explicit information on the periodic solution. 
Here, we use the spectral theory of~TP matrices to analyze   trajectories
of the~$T$-periodic 
   TPDS~\eqref{eq:tpds} in the vicinity of a periodic solution (and  in the special case of a time-invariant vector field, near an equilibrium). 
   Our main result characterizes the sign pattern of  the ``most stable''
and ``most unstable'' perturbation directions near a periodic solution. We also provide an intuitive explanation for the structure of these sign patterns. 
We   demonstrate an application of 
 these theoretical  
 results to an important model from systems biology called the ribosome flow model~(RFM)~\cite{rfm_chap}.

We note that when~$n=2$ and~\eqref{eq:tpds} is a  strictly  cooperative system 
then it is also a TPDS (as a~$2\times 2$ matrix is always tri-diagonal), so our results hold in this special case as well.

The remainder of this note  is organized as follows.
The next section briefly reviews several  known definitions and results that are used
later on. Section~\ref{sec:main} presents our main results, and  describes an application of these results to the~RFM.
The final section concludes. 
We use standard notation. Vectors  [matrices] are denoted by small [capital] letters. 
For~$A\in \R^{n\times m}$, $A^T$ is the transpose of~$A$.

 
\section{Preliminaries}
We begin by briefly reviewing
 known results that will be used later on.

\subsection{Sign variation diminishing property}
An important property of TP matrices is that multiplying a vector by a~TP matrix can never increase  the number of sign variations. 
To explain this, we recall two definitions for the number of sign changes in a   vector. For~$z\in\R^n\setminus\{0\}$, 
let~$s^-(z)$ denote the number of sign changes in the vector~$z$, after deleting all the zero entries (with~$s^-(0)$ defined as zero). For example,~$s^-(\begin{bmatrix}   1& -2 &0&0&3 \end{bmatrix}^T)=2$. 
For~$z\in\R^n$, let~$s^+(z)$ denote the maximal possible number of sign changes in the vector~$z$, after replacing every zero entry by either plus one or minus one. 
For example,~$s^+(\begin{bmatrix}   1& -2 &0&0&3 \end{bmatrix}^T)=4$. Note that these 
definitions imply that
\be\label{eq:smsp}
0\leq s^-(z)\leq s^+(z)\leq n-1,\text{ for any }z\in\bbr^n.
\ee

The next result describes the sign variation diminishing
property of TP matrices. 
\begin{proposition}\label{prop:svdp}\cite{total_book}
Let~$A\in\R^{n \times n}$ be TP. Then
$
s^+(Ax)\leq s^-(x)$  for all $ x\in\R^n\setminus\{0\}$.
\end{proposition}

\subsection{Oscillatory matrices}
  A   matrix~$A\in\bbr^{n\times n}$ is called oscillatory if it is~TN and there exists an integer~$k>0$ such that~$A^k$ is~TP~\cite{gk_book}. 
	The smallest such~$k$ is called the exponent of the oscillatory matrix~$A$~\cite{exp_osci}.
	A TN matrix is oscillatory if and only if~(iff) it is non-singular and irreducible~\cite{gk_book}. Oscillatory matrices have special  
spectral properties: their eigenvalues are all real, simple, and positive, and the corresponding eigenvectors have a special sign pattern~\cite{pinkus}.
The next result summarizes this   
spectral structure. 
\begin{proposition}\cite{pinkus}\label{prop:specs}
If~$A\in\bbr^{n\times n}$ is oscillatory then all its eigenvalues are real, positive, and simple. 
Denote these eigenvalues as~$\alpha_1>\alpha_2>\dots>\alpha_n>0$, and let~$v^k\in\bbr^{n}\setminus\{0\}$ denote an eigenvector corresponding to~$\alpha_k$. For any~$1\leq i\leq j\leq n$,
let~$V^{ij}:=\spanop \{v^i,v^{i+1},\dots,v^j\}$. Then  
\[
										i-1\leq s^-(z)\leq s^+(z) \leq j-1 \text{ for any } z\in V^{ij}\setminus\{0\}.
\]
\end{proposition}

\begin{Remark}
Note that this implies in particular that
\be\label{eq:zewy}
	  s^-(v^i) = s^+(v^i) =i-1,\text{ for any } i\in\{1,\dots,n\}.
\ee
In general, the eigenvectors may include zero entries. However, the zeros must be  located such that~\eqref{eq:zewy} holds. 
In particular,~$v^1$ cannot include any zero entries because~$s^+(v^1)=0$,
and~$v^n$ cannot include any zero entries because~$s^-(v^n)=n-1$.
\end{Remark}

 \begin{Example}
Let~$A=\begin{bmatrix}  3& 2 &0 \\ 1& 3& 1\\ 0.1&1 &4\end{bmatrix}$. 
All the $1\times 1$ minors of~$A$ (i.e., its entries) are non-negative. The~$2\times 2$ minors 
are~$7, 3, 2, 2.8, 12, 8, 0.7, 3.9, 11$, 
and the single~$3\times 3$ minor is~$\det(A)=25.2$. Thus,~$A$ is~TN. 
Calculating all the  minors of~$A^2$ shows that they are all positive, so~$A^2$ is~TP. Thus,~$A$ is oscillatory. 
Its eigenvalues  are
$\alpha_1=5.03851$, $\alpha_2=  3.55435 $, $\alpha_3=1.40714$, (all numerical values in this paper are to 5-digit accuracy)
with corresponding   eigenvectors:  
$v^1=\begin{bmatrix} 0.55898& 0.569742& 0.602442     \end{bmatrix}^T $,
$v^2=\begin{bmatrix} 
0.746782 & 0.206989& -0.632038\end{bmatrix}^T $,
and $v^3=\begin{bmatrix}   0.765516& -0.609679& 0.205614  \end{bmatrix}^T $.
Note that~$s^-(v^i)=s^+(v^i)=i-1$, $i\in\{1,2,3\}$.
\end{Example}

For our purposes, it is important to know when a tri-diagonal matrix is an oscillatory matrix. 
\subsection{Tri-diagonal oscillatory matrices}
Consider the~$n\times n$ tri-diagonal  matrix
\be\label{eq:defT}
T(a_i,b_i,c_i):=
 \begin{bmatrix}
a_1 & b_1 & 0 &\dots& 0 & 0& 0\\
c_1 & a_2 & b_2 &\dots& 0 & 0& 0\\
0&c_2 & a_3  &\dots& 0 & 0& 0\\
\vdots &\vdots & \vdots & \ddots & \vdots & \vdots & \vdots\\
0 & 0 & 0 &\dots& c_{n-2} & a_{n-1}& b_{n-1}\\
0 & 0 & 0  &\dots& 0 & c_{n-1}& a_n
\end{bmatrix}.
\ee
Assume that~$b_i,c_i\geq 0$, and that the entries satisfy the \emph{dominance condition}:
\be\label{eq:domcon}
a_i \geq b_i + c_{i-1},\quad i=1,\dots,n,
\ee
with~$b_n:=0$, and~$c_0:=0$. 
Then all the minors of~$T$ are non-negative, i.e.~$T$  is~TN~\cite[Chapter~0]{total_book}.
If, furthermore,~$T$ is non-singular and irreducible 
then it is oscillatory~\cite{gk_book}.

\subsection{Linear TPDSs}
Schwarz~\cite{schwarz1970}
considered the system
\be\label{eq:lintpds}
\dot x(t)=A(t)x(t),
\ee
with~$t\to A(t)$  continuous.
The system is called a linear TPDS
if the corresponding transition matrix~$\Phi(t,t_0)$ (i.e, the matrix such that~$x(t)=\Phi(t,t_0)x(t_0)$ for all~$t\geq t_0\geq0$
and~$x(t_0) \in \R^n$)  is~TP for any~$t>t_0\geq 0$. 
Schwarz showed that this holds iff: 1)~$a_{ij}(t)\equiv 0$ for all~$|i-j|>1$,
2)~$a_{ij}(t)\geq 0$ for all~$|i-j|=1$, and 
3)~any of the functions in~2) does not vanish  identically on any interval of positive length. 
In particular, if~$A(t)\equiv A$ is constant then the system is a TPDS iff~$A\in \M^+$. 
One important property of a linear~TPDS, that follows
 from Prop.~\ref{prop:svdp}, is that for any non-trivial solution~$x(t)$,
\be\label{eq:seinw}
s^+( x(t_2) ) \leq s^- (x(t_1)) , \text{ for all } t_2>t_1\geq 0.  
\ee
Thus, the number of sign variations along a solution of the system is   non-increasing. Furthermore,  it can be shown  that~$s^-(x(t))=s^+(x(t))$
except perhaps at up to~$n-1$ isolated points~\cite{fulltppaper}.

One implication of this is that if~$\gamma(t)$ is a $T$-periodic solution of a linear TPDS  (that is not the trivial solution~$\gamma(t)\equiv 0$) 
then 
\[
s^-(\gamma(t)) =  s^+(\gamma(t))=  s^-(\gamma(0 ))  \text{ for all } t\geq 0.
\]


 The next section describes our main results.
 It is useful to begin with  case of a linear TPDS, and then consider the nonlinear case.
\section{Main Results}\label{sec:main}
 
\subsection{  $T$-periodic linear  TPDS}

Consider the linear-time varying~(LTV) system~\eqref{eq:lintpds}
with~$A(t)$   a continuous and~$T$-periodic matrix. 
We also assume that~$A(t)\in\M^+$ for all~$t\in[0,T)$, so~\eqref{eq:lintpds} is a linear~TPDS. 
Then any trajectory with an initial condition in~$\Omega$ converges to a $T$-periodic 
solution of~\eqref{eq:lintpds} (that is not necessarily unique).

The transition matrix satisfies~$\Phi(0)=I$  and  by Floquet theory~\cite{chicone_2006},
$
\Phi(t+T)= \Phi(t) \Phi(T)$,  for all~$t\geq 0$.
 Let~$B:=\Phi(T)$.
 The eigenvalues of~$B$ are called the
 characteristic multipliers of~\eqref{eq:lintpds}. Since~$A(t)\in \M^+$, $\Phi(t)$ is TP for any~$t>0$ and in particular 
 $B$ is~TP. Let
 \be\label{eq:lamp}
 \lambda_1>\lambda_2>\dots>\lambda_n>0,
 \ee
 denote the eigenvalues of~$B$, and let~$v^i\in\R^n$ denote the eigenvector corresponding to~$\lambda_i$. 
 
 Let~$\gamma(t)$ denote a $T$-periodic solution of~\eqref{eq:lintpds} with~$\gamma(0)\not = 0$. 
Then~$\gamma(T)=B\gamma(0)$ and since~$\gamma(T)=\gamma(0)$, this implies that 
there exists~$p\in\{1,\dots,n\}$ such that
 \be\label{eq:lamporef}
 \lambda_1>\lambda_2>\dots>\lambda_p=1>\dots > \lambda_n>0,
 \ee

 Fix a ``perturbation direction''~$w\in\R^n$ and consider the difference
 \be\label{eq:zz}
 z(t):=x(t,\gamma(0)+w)-x(t,\gamma(0) ). 
 \ee
 Then for any integer~$k>0 $,
 \begin{align*}
                        z(kT)&=\Phi(kT)z(0)\\
                        &=B^kz(0) \\
                        &=   B^k w. 
 \end{align*}
 Let~$c_i\in\R$ be such that~$w=\sum_{i=1}^n c_i v^i$. Then 
\be\label{eq:zz2}
z(kT) = \sum_{i=1}^n c_i \lambda_i^k v^i.
\ee
Combining this with~\eqref{eq:lamporef}
implies  the following result. 
\begin{proposition}\label{prop:linearcase}
Consider the linear TPDS~\eqref{eq:lintpds}, and assume that it   admits a~$T$-periodic solution~$\gamma(t)$ with~$\gamma(0)\not = 0$.
There are two possible cases.

\noindent Case 1. If~$\lambda_1>1$ then the periodic solution is not orbitally  stable, 
$v^1$   is the ``worst''  perturbation direction
in the sense that
this perturbation takes the state  away from~$\gamma$ as quickly  as possible. 
Also,~$\lambda_n \leq  1$, and if~$\lambda_n<1$ then~$v^n$  is the ``best''  perturbation direction  in the sense that
this perturbation takes the state  back to~$\gamma$  as  quickly    as possible.

\noindent  Case 2. If~$\lambda_1=1$ then~\eqref{eq:lamporef} implies that
the   periodic solution is stable and orbitally asymptotically stable, 
and that~$v^2$ [$v^n$] is the ``worst'' [``best''] perturbation direction
in the sense that
this perturbation takes the state  back to~$\gamma$   as slowly [quickly] as possible. 
\end{proposition}

Note  that for  a TPDS, even if we do not have an explicit expression for~$\gamma(t)$,
we always know that the eigenvalues~$\lambda_i$ of~$B=\Phi(T)$ are real, positive,  and distinct, and also the
special    sign pattern of each of the eigenvectors~$v^i$.

 The next simple example demonstrates Prop.~\ref{prop:linearcase}.
 
 \begin{Example}
 Consider~\eqref{eq:lintpds}
 with~$n=2$ and
 \[
 A(t)=\begin{bmatrix}
  -2&2+\sin(t)  \\ 2+\sin(t) &-2
 \end{bmatrix}.
 \]
 Note that~$A(t)\in\M^+$ for all~$t\geq 0$, and that~$A(t)$ is~$T$-periodic for~$T=2\pi$.
The transition matrix is 
 \[
 \Phi(t)=\exp(-2t)\begin{bmatrix}
 \cosh(c(t))& \sinh(c(t))\\  \sinh(c(t))&\cosh(c(t))
 \end{bmatrix},
 \]
with~$c(t):= 2 t - \cos(t)+1$.  Note that every entry of~$\Phi(t)$ is positive and that~$\det(\Phi(t))>0$ (i.e. $\Phi(t)$ is TP) for all~$t> 0$. Here,
\[
B=\Phi(2\pi)= \exp(-4\pi) \begin{bmatrix}
 \cosh( 4 \pi) & \sinh( 4 \pi )\\  \sinh( 4 \pi )&\cosh( 4 \pi)  
 \end{bmatrix}.
\]
The eigenvalues of~$B$ are
\[
\lambda_1=1 ,\quad \lambda_2=\exp(-8 \pi),
\]
with corresponding eigenvectors
\[
v^1=\begin{bmatrix}1&1 \end{bmatrix}^T ,\quad v^2=\begin{bmatrix}-1&1 \end{bmatrix}^T .
\]
Hence, a~$2\pi$-periodic solution is
\begin{align*}
\gamma(t)&:=x(t,v^1) \\
&= \Phi(t)v^1\\
&= \exp(-2t) ( \cosh(c(t))+\sinh(c(t))  )v^1\\
&=\exp( 1-\cos(t) )v^1.
\end{align*}
As expected,~$v^1$ [$v^2$] has zero [one] sign changes.
To demonstrate~\eqref{eq:zz2} in this case, 
define~$z$ as in~\eqref{eq:zz}.  
Any~$w\in\R^2$ can be written as
$
w=\frac{w_1+w_2}{2} v^1+ \frac{w_2-w_1}{2} v^2, 
$
so
\begin{align*}
z(kT)&= B^k w\\
&= \frac{w_1+w_2}{2}    v^1 +  \frac{w_2-w_1}{2}\exp(-8k \pi )  v^2.    
\end{align*}
 \end{Example}

\subsection{$T$-periodic nonlinear TPDS}

Consider the nonlinear system~\eqref{eq:tpds} with a~$T$-periodic vector field.
Assume that its trajectories evolve on  the state-space~$\Omega \subseteq \bbr ^n $, and that
\be \label{eq:oft}
J(t,z)\in \M^+ \text{ for all } t\in[0,T), z \in \Omega,
\ee
that is,~\eqref{eq:tpds}  is a TPDS.
 
 We first consider the case where~$f$ is time-invariant (and hence~$T$-periodic for any~$T$). 
The next result describes 
 the special spectral structure of~$J(z)$ at any point~$z\in\Omega$.

\begin{proposition}
\label{thm:jspectpds}
Fix~$z\in \Omega$.  
The eigenvalues of~$J(z)$ are real and simple. Denote these eigenvalues 
by
\[
\alpha_1(z)>\alpha_2(z)>\dots>\alpha_n(z) ,
\]
 and let~$v^k(z)\in\bbr ^n $ denote the eigenvector corresponding to~$\alpha_k(z)$. For any~$1\leq i\leq j\leq n$, let~$V^{ij}(z):=\spanop \{v^i(z),v^{i+1}(z),\dots,v^j(z)\}$.
 Then 
\be\label{eq:speiv}
										i-1\leq s^-(y)\leq s^+(y) \leq j-1,\quad \text{for any } y\in V^{ij}(z)\setminus\{0\}.
\ee
\end{proposition}
\begin{proof}
Fix~$z\in  \Omega$. Then~$J(z)$ is a Jacobi matrix.
Pick~$s>0$ large enough so   that~$sI+J(z) $ satisfies the dominance condition~\eqref{eq:domcon}, and is nonsingular. 
 Then~$sI+ J(z)$ 
is~TN, nonsingular and irreducible, so it is oscillatory. 
Applying Prop.~\ref{prop:specs}  completes the proof.
\end{proof}

In particular, if~$e$ is an equilibrium then the  set of eigenvalues and eigenvectors of~$J(e)$
provides a rather complete picture of the dynamical behavior near~$e$. 
Indeed, the  Hartman–Grobman theorem~\cite{chicone_2006} asserts that 
  the phase portrait near~$e$ (assuming that~$J(e)$
	has no eigenvalues with a zero real part) 
 is the same as the phase portrait of the LTI  system~ $\dot y=J(e)y$, up to a continuous change of coordinates~$x=p(y)$, with $p:\bbr^n\to\bbr^n$ satisfying~$p(0)=e$.
  
Suppose for example that all the eigenvalues are negative. Then the eigenvector~$v^1(e)$ [$v^n(e)$]
corresponds to the ``slowest'' [``fastest'']
 eigenvalue~$\alpha_1(e)$ [$\alpha_n(e)$]. 
 A perturbation of the steady-state in the form~$e\to p(\varepsilon v^1(e))$ [$e\to p(\varepsilon v^n(e))$], with~$|\varepsilon|$ small,
will be ``compensated'' at the slowest [fastest] possible rate.

\begin{Example}
Consider the system:
\begin{align}\label{eq:neuro}
\dot x_1&=-2x_1+\tanh(x_1)+2\tanh(x_2),\nonumber \\
\dot x_2&=- x_2+ (1/2) \tanh(x_1)+ \tanh(x_2).
\end{align}
Such  systems appear in models of neural networks with~$\tanh(\cdot)$ as the  activation function. 
It is straightforward to verify  
that there are three equilibrium points in~$\R^2$:
$e^1=\begin{bmatrix} 0&0 \end{bmatrix}^T$,  
$e^2= \begin{bmatrix} 1.28784&1.28784 \end{bmatrix}^T$,
and~$e^3= -e^2$.
The Jacobian  of~\eqref{eq:neuro} is
\begin{align}\label{eq:jacon}
J(x)=\begin{bmatrix} 
					-2+\cosh^{-2}(x_1) & 2\cosh^{-2}(x_2)\\
			(1/2)	\cosh^{-2}(x_1)& -1+  \cosh^{-2}(x_2) 
\end{bmatrix},
\end{align}
so~$J(x) \in \M^+$ for all~$x\in\R^2$. 
The eigenvalues of~$J(e^1)$ 
are
\[
  (\sqrt{5}-1)/2, \quad  (-\sqrt{5}-1)/2,
\]
with corresponding eigenvectors
\[ 
\begin{bmatrix} \sqrt{5}-1 & 1  \end{bmatrix}^T, \quad 
 \begin{bmatrix} -\sqrt{5}-1 & 1 \end{bmatrix}^T.
\] 
The eigenvalues of~$J(e^2)$ (and of~$J(e^3)$) 
are
\[
-0.672232,\quad   -1.80202,   
\]
with corresponding eigenvectors
\[
\begin{bmatrix}  0.442698 & 0.896671   \end{bmatrix}^T,\quad
 \begin{bmatrix}  0.992469 & -0.122499    \end{bmatrix}^T.
\] 
As expected, the eiegnvalues are real and distinct,
and the eigenvectors have the specified sign pattern.~\hfill{$\square$}
\end{Example}

We now turn to consider the case where the nonlinear system is time-varying, and~$T$-periodic with a minimal period~$T>0$. 
For~$x_0 \in \Omega$, let
\[
H(t,x_0):=\frac{\partial}{\partial x_0 }x(t,x_0),
\]
that is, $H$ maps a change 
in the initial condition at time~$0$ to the change in 
the solution at time~$t$. Then
$
H(0,x_0)=I,
$
and
\begin{align}\label{eq:potdd}
 \dot H(t,x_0)&:=\frac{d}{dt} H(t,x_0) \nonumber\\
 &= \frac{\partial}{\partial x_0 }f(t,x(t,x_0))\nonumber\\
 &=J(t,x(t,x_0)) H(t,x_0).
\end{align}
 Since~$J \in M^+$, this is a linear time-varying TPDS. Let~$a:=
 \gamma'(0) $,
 where~$\gamma(t)$ is a~$T$-periodic solution  of~\eqref{eq:tpds}. Then for~$x_0=a$ the linear TPDS~\eqref{eq:potdd} is also~$T$-periodic, so~$ \Phi(t): = H(t,a) $ satisfies
  $
  \Phi(t+T)=\Phi(t)B,
  $
with~$B:=\Phi(T)$. We conclude that
\[
\frac{\partial} {\partial x_0} 
x(kT,a)=B^k .
\]

For~$\varepsilon>0$ 
and~$\omega \in \R^n$,  let
$
z(t):=x(t,a+\varepsilon \omega)-x(t,a).
$
Then for a fixed time~$t\geq 0$,
$
z(t)= \varepsilon \Phi(t) \omega +o(\varepsilon),
$
so for a fixed~$k$,
\begin{align*}
    z(k T)&= \varepsilon \Phi(k T) \omega +o(\varepsilon) = \varepsilon B^k \omega +o(\varepsilon).
\end{align*}
Thus, as in the cases described above,   $B$ always has $1$ as an eigenvalue and a corresponding eigenvector~$a$. The eigenvalues and eigenvectors of the TP matrix~$B$ determine the response to a small perturbation~$a\to a+\varepsilon \omega $.


The special  sign pattern of the eigenvectors~$v^1$ and~$v^n$ of~$B$  
can be explained intuitively using the cooperative and tridiagonal 
structure of the Jacobian  of a~TPDS. 
We demonstrate this using a model from systems biology called the ribosome flow model~(RFM). 

\subsection{Application: the RFM}
The RFM  
is a phenomenological  model for the flow of ``material'' along a~1D chain composed of~$n$ consecutive sites. 
The RFM is the dynamic mean field approximation of a 
fundamental  model from
statistical mechanics called the \emph{totally asymmetric simple exclusion process}~(TASEP)~\cite{solvers_guide,Kriecherbauer_2010}. This model describes the stochastic motion of  
particles hopping along  a~1D  chain of sites while restricted
   by the simple exclusion principle, namely,  two particles
cannot occupy the same site at the same time.

The RFM, and more generally networks of interconnected~RFMs, has been extensively used   to model
the flow of ribosomes along the mRNA during translation~\cite{rfm_concave,RFM_feedback,RFM_stability,RFMR,RFMEO,RFM_model_compete_J,rfmnets,rfmdr,mexdenrfm,RFM_RANDOM}, and  also
other important processes such as   the transfer of 
the phosphoryl group from the sensor kinases to the ultimate target during phosphorelay~\cite{EYAL_RFMD1}.

  The RFM  is a set of~$n$ ODEs:
\begin{equation}\label{rfm}
\dot x_i(t)= \lambda_{i-1}x_{i-1}(t)(1-x_{i}(t))-\lambda_{i}x_{i}(t)(1-x_{i+1}(t)), 
\end{equation}
with $i=1,\dots,n$, $x_{0}(t)\equiv 1$, and $x_{n+1}(t)\equiv 0$.
The state-variable~$x_i(t)$ represents the density at site~$i$ at time~$t$, normalized such that~$x_i(t)=0$
[$x_i(t)=1$] corresponds to site~$i$ being empty [full] at time~$t$. 
The positive parameter~$\lambda_{i}$ represents the transition rate from site~$i$ to site~$i+1$. 
In particular, $\lambda_0$ [$\lambda_n$] 
is called the initiation [exit] rate, and the other~$\lambda_i$s are the elongation rates. 
 
To explain~\eqref{rfm}, consider the equation for~$\dot x_2$, namely,
\[
\dot x_2 =\lambda_{1}x_{1} (1-x_{2} )-\lambda_{2}x_{2} (1-x_{3} ).
\]
This asserts that the change in the density at site~$2$ 
is the flow from site~$1$ to site~$2$ 
minus the flow from site~$2$ to site~$3$. 
The flow from site~$1$ to site~$2$, given by~$\lambda_{1}x_{1} (1-x_{2} )$,
is proportional to the density at site~$1$, the amount of ``free space''~$(1-x_2 )$
at site~$2$, and the transition rate~$\lambda_{1}$ from site~$1$ to site~$2$. 
In particular, as the density in site~$2$ increases, the flow from site~$1$ to site~$2$
decreases. This is a ``soft version'' of the simple exclusion principle in~TASEP.

An important property of the RFM, ``inherited'' from~TASEP,
 is that allows to study the generation of ``traffic jams'' along the chain.
Indeed, if some~$\lambda_i$ is very small w.r.t. the other rates then 
the exit rate from site~$i$  will be small. Then the density at site~$i$ will increase, and consequently, the transition rate from site~$i-1$  will decrease. In this way, a traffic jam of
high density sites will evolve ``behind'' site~$i$. 
The dynamic evolution and  implications of   traffic jams of ``biological particles'' like ants, ribosomes, and  molecular motors   
 are  attracting  considerable interest (see, e.g.~\cite{Ross5911,tuller_traffic_jams2018,neurojams,ants_jams}).

The output flow from the last site is~$R(t):=\lambda_n x_n(t)$. 
This represents the rate at which ribosomes exit the mRNA, i.e. the protein production rate.

The state space of RFM is the unit
cube $ [0,1]^n$. For~$a \in[0,1]^n$, let~$x(t,a)$ denote the solution  of the~RFM
at time~$t$  with~$x(0)=a$. The flow possesses a unique  globally
 stable~(GAS) equilibrium~$e\in  ( 0,1) ^n$~\cite{RFM_stability,cast_book}, that is,
$
\lim_{t \to\infty }
x(t,a)=e$,   for all~$a\in[0,1]^n$.
This steady-state represents a set of densities~$e_1,\dots,e_n$ for which 
  the flow into each site is equal to the flow out of this site.
	In  physics, this is sometimes referred to as a nonequilibrium stationary state.	
	
	Let~$R:=\lambda_n e_n$. 
This	is the  flow 
  of ribosomes out of the chain, and thus the protein production rate,
at equilibrium.
It follows from~\eqref{rfm} that
\begin{equation}\label{steady_rfm}
\lambda_{i}e_{i}(1-e_{i+1})=R,\quad i=0,\dots,n, 
\end{equation}
where~$e_0:=1$ and~$e_{n+1} :=0$.

Let~$J(x)$ denote the Jacobian of the vector field in the~RFM.
Then~$J(x)$ admits the tridiagonal structure~$J(x)=T(a_i(x),b_i(x),c_i(x))$ in~\eqref{eq:defT} 
with
$
a_i(x)=-\lambda_{i-1} x_{i-1} -\lambda_i(1-x_{i+1})$, $
b_i(x) =  \lambda_i x_i $, and $  
c_i(x)  =\lambda_i (1-x_{i+1}) $, 
where~$x_0:=1$, and~$x_{n+1}:=0$. Thus,  the RFM is a~TPDS on the invariant set~$ (0,1)^n$, and
  Theorem~\ref{thm:jspectpds} implies that  the eigenvalues of~$J(e)$ are 
real and simple and the corresponding eigenvectors satisfy the sign pattern~\eqref{eq:speiv}.
Furthermore, since the~RFM is contractive~\cite{cast_book},  
the eigenvalues of~$J(e)$ are all 
real, simple, and negative. 

 The sign pattern of the eigenvectors corresponding to the slowest [fastest] eigenvalue~$\alpha_1(e)$ 
[$\alpha_n(e)$] 
  can be explained as follows. Recall that~$v^1(e)$
has zero sign variations. We may assume that~$v^1_i(e)>0$ for all~$i$. Then $e\to e+\varepsilon v^1(e)$ is difficult to compensate  because it increases all the densities in the chain. This corresponds a to very congested situation, and returning to the steady-state~$e$ takes more time.  

On the other-hand, a perturbation in the form~$e\to e+\varepsilon v^{n}(e)$ is easier to compensate  because~$v^{n}(e)$ has an alternating sign pattern, that is, we may assume that~$v^{n}_i(e)>0$  for~$i$ odd, and
$v^{n}_i(e)<0$ for~$i$ even.  This represents  a positive addition to~$e_1$, a negative addition to~$e_2$, a positive addition to~$e_3$, and so on. 
The~RFM dynamics compensates for this variation quite naturally: the additional density in an odd site   generates an increased flow to the consecutive site and in this way also ``automatically'' takes care of 
the lower  densities in the even sites.

The next example considers specific values for the rates for which~$e$ (and thus~$J(e)$)
 is known explicitly  for any~$n$.
\begin{Example}\label{exa:mavp}
Consider the RFM with~$\lambda_0=\lambda_n=1/2$,
 and~$\lambda_i=1$ for~$i=1,\dots,n-1$. 
It can be shown using the spectral representation of~$e$~\cite{rfm_concave} 
  that in this case~$e_i=1/2$, $i=1,\dots,n$. This implies that~$J(e)=T(a_i(e),b_i(e),c_i(e))$ 
 is an~$n\times n $ Toeplitz tridiagonal  matrix with~$-1$ on the main diagonal, and~$1/2$ on the sub- and super-diagonals. 
It is well-known~\cite{Yueh2005EIGENVALUESOS} that the   eigenvalues of such a matrix  
are
\[
\alpha_k=-1+\cos(k\pi /(n+1)) , \quad k=1,\dots,n,
\] 
and the eigenvector  corresponding to~$\alpha_k $ is
$
v^k=\begin{bmatrix}  \sin( \frac{k \pi}{n+1} ) & \sin(\frac{2k   \pi}{n+1} ) &\dots &\sin( \frac{n k   \pi}{n+1} )   \end{bmatrix}^T 
$.
In particular,
\[
v^1:=\begin{bmatrix}  \sin( \frac{ \pi}{n+1} ) & \sin(\frac{2   \pi}{n+1} ) &\dots &\sin( \frac{n \pi}{n+1} )   \end{bmatrix}^T. 
\]
and
\[
v^n:=\begin{bmatrix}  \sin( \frac{n \pi}{n+1} ) & \sin(\frac{2n   \pi}{n+1} ) &\dots &\sin( \frac{n^2   \pi}{n+1} )   \end{bmatrix}^T. 
\]

As expected, all the eigenvalues are real, negative and simple,   all the entries of~$v^1$ are positive,
and the entries of~$v^n$ have the sign pattern~$(+,-,+,-,\dots)$. 
The maximal eigenvalue is
$
\alpha_1=-1+\cos(\pi /(n+1)) .
$ Thus,  
\begin{align}\label{eq:asym}
\lim_{n\to \infty} \frac{\log(-\alpha_1(n))}{\log(n)} &=
\lim_{n\to \infty} \frac{\log(  
\frac{\pi^2}{2(n+1)^2}-\frac{\pi^4}{24(n+1)^4}+\dots )}{\log(n)}\nonumber  \\
&=-2.
\end{align}
Thus, the relaxation time of the system   behaves like~$n^2$ for large~$n$.

The minimal  eigenvalue is
$
\alpha_n=-1+\cos(\pi n /(n+1)) ,
$
so
\begin{align*}
\lim_{n\to \infty} \alpha_n   &=-2
\end{align*}
Thus, the fastest convergence rate to~$e$, corresponding to a perturbation in the form~$e+\varepsilon v^{n }$,  is independent of~$n$ for large~$n$.~\hfill{$\square$}
\end{Example}

\begin{Remark}
The asymptotic behavior described in~\eqref{eq:asym}
 seems to hold  for  other rates as well. For example,
Fig.~\ref{fig:maxee} depicts $ \log(-\alpha_1(n))$ as a function of~$\log(n) $ 
 for the case where~$\lambda_i=1$ for~$i\in\{0,\dots,n\}$.
It may be seen that for large values of~$n$ the graph behaves like a line with slope~$-2$.
\end{Remark}

\begin{figure}[t]
 \begin{center}
   \includegraphics[scale=0.4]{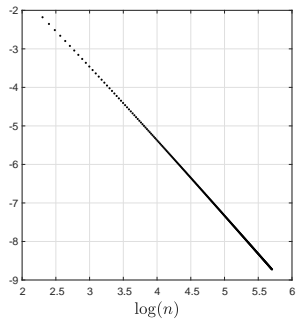}
\caption{ $ \log(-\alpha_1(n)) $ as a function of~$\log(n)$.  }\label{fig:maxee}
\end{center}
\end{figure}

\section{Conclusion}
Any bounded solution
of a~$T$-periodic TPDS converges to a~$T$-periodic solution of the TPDS. 
If the~$T$-periodic vector field represents  a $T$-periodic excitation then this implies that the dynamical system entrains to the excitation. 
In particular, any bounded solution of  
a time-invariant  TPDS converges to an equilibrium~\cite{fulltppaper}.

It is important to
understand how a perturbation  from a  periodic solution affects the dynamics and, in particular, what  are the convergence or divergence rates associated with different perturbation directions.
In general, Floquet theory  does not provide explicit answers to theses questions.  
We used the spectral theory of TP matrices
to analyze this question for TPDSs and showed that the ``best'' and ``worst'' perturbation directions have special sign patterns that admit an intuitive interpretation.

An interesting topic for future research 
is to use this information to design appropriate 
control algorithms for~TPDSs.


\end{document}